\documentclass[11pt]{amsart}
\usepackage{amscd}
\usepackage{amsmath}
\usepackage{amsxtra}
\usepackage{amsfonts}
\usepackage{amssymb}

\newtheorem{theorem}{Theorem}[section]
\newtheorem{corollary}[theorem]{Corollary}
\newtheorem{lemma}[theorem]{Lemma}
\newtheorem{proposition}[theorem]{Proposition}

\theoremstyle{definition}
\newtheorem{definition}[theorem]{Definition}
\newtheorem{remark}[theorem]{Remark}

\theoremstyle{remark}

\renewcommand{\theclaim}{\textup{\theclaim}}

\newtheorem*{acknowledgements}{Acknowledgements}

\numberwithin{equation}{section}

\def\openone%{\hbox{\upshape \small1\kern-3.3pt\normalsize1}}

{\mathchoice

{\hbox{\upshape \small1\kern-3.3pt\normalsize1}}

{\hbox{\upshape \small1\kern-3.3pt\normalsize1}}

{\hbox{\upshape \tiny1\kern-2.3pt\SMALL1}}

{\hbox{\upshape \Tiny1\kern-2pt\tiny1}}}

\makeatletter

\newbox\ipbox

\newcommand{\ip}[2]{\left\langle #1\, , \,#2\right\rangle}
\newcommand{\diracb}[1]{\left\langle #1\mathrel{\mathchoice

{\setbox\ipbox=\hbox{$\displaystyle \left\langle\mathstrut
#1\right.$}

\vrule height\ht\ipbox width0.25pt depth\dp\ipbox}

{\setbox\ipbox=\hbox{$\textstyle \left\langle\mathstrut
#1\right.$}

\vrule height\ht\ipbox width0.25pt depth\dp\ipbox}

{\setbox\ipbox=\hbox{$\scriptstyle \left\langle\mathstrut
#1\right.$}

\vrule height\ht\ipbox width0.25pt depth\dp\ipbox}

{\setbox\ipbox=\hbox{$\scriptscriptstyle \left\langle\mathstrut
#1\right.$}

\vrule height\ht\ipbox width0.25pt depth\dp\ipbox}

}\right. }

\newcommand{\dirack}[1]{\left. \mathrel{\mathchoice

{\setbox\ipbox=\hbox{$\displaystyle \left.\mathstrut
#1\right\rangle$}

\vrule height\ht\ipbox width0.25pt depth\dp\ipbox}

{\setbox\ipbox=\hbox{$\textstyle \left.\mathstrut
#1\right\rangle$}

\vrule height\ht\ipbox width0.25pt depth\dp\ipbox}

{\setbox\ipbox=\hbox{$\scriptstyle \left.\mathstrut
#1\right\rangle$}

\vrule height\ht\ipbox width0.25pt depth\dp\ipbox}

{\setbox\ipbox=\hbox{$\scriptscriptstyle \left.\mathstrut
#1\right\rangle$}

\vrule height\ht\ipbox width0.25pt depth\dp\ipbox}

} #1\right\rangle}

\newcommand{\cj}[1]{\overline{#1}}

\newcommand{\bz}{\mathbb{Z}}

\newcommand{\br}{\mathbb{R}}

\newcommand{\bt}{\mathbb{T}}
\newcommand{\bn}{\mathbb{N}}

\def\blfootnote{\xdef\@thefnmark{}\@footnotetext}

\input xy
\xyoption{all}
\usepackage{amssymb}

\def\T{\mathcal{T}}

\def\F{\mathcal{F}}

\def\H{\mathcal{H}}

\def\-{^{-1}}

\def\ty{\emptyset}

\def\C{\mathbb{C}}
\def\Z{\mathbb{Z}}

\begin{document}

\title[Wavelet representations and Martin bondaries]{Decomposition of wavelet representations and Martin boundaries}
\author{Dorin Ervin Dutkay}
\blfootnote{Swedish Foundation for International Cooperation in Research and Higher Education (STINT) and Swedish Research Council (grant 2007-6338).}
\address{[Dorin Ervin Dutkay] University of Central Florida\\
	Department of Mathematics\\
	4000 Central Florida Blvd.\\
	P.O. Box 161364\\
	Orlando, FL 32816-1364\\
U.S.A.\\} \email{ddutkay@mail.ucf.edu}

\author{Palle E.T. Jorgensen}
\address{[Palle E.T. Jorgensen]University of Iowa\\
Department of Mathematics\\
14 MacLean Hall\\
Iowa City, IA 52242-1419\\}\email{jorgen@math.uiowa.edu}

\author{Sergei Silvestrov}
\address{[Sergei Silvestrov] Centre for Mathematical Sciences\\ 
Lund University\\
Box 118, SE-221 00 Lund, Sweden\\
and Division of Applied Mathematics\\ 
School of Education, Culture and Communication\\ M{\"a}lardalen University, Box 883, 721 23\\ V{\"a}ster{\aa}s, Sweden
}
\email{Sergei.Silvestrov@math.lth.se}

\thanks{} 
\subjclass[2000]{65T60,47A67,31C35}
\keywords{irreducible representation, wavelet, Martin boundary, harmonic function}

\begin{abstract}
 We study a decomposition problem for a class of unitary representations associated with wavelet analysis, wavelet representations, but our framework is wider and has applications to multi-scale expansions arising in dynamical systems theory for non-invertible endomorphisms.

Our main results offer a direct integral decomposition for the general wavelet representation, and we solve a question posed by Judith Packer. This entails a direct integral decomposition of the general wavelet representation. We further give a detailed analysis of the measures contributing to the decomposition into irreducible representations.  We prove results for associated Martin boundaries, relevant for the understanding of wavelet filters and induced random walks, as well as classes of harmonic functions.

   Our setting entails representations built from certain finite-to-one endomorphisms $r$ in compact metric spaces $X$, and we study their dilations to automorphisms in induced solenoids. Our wavelet representations are covariant systems formed from the dilated automorphisms. They depend on assigned measures $\mu$ on $X$. It is known that when the data $(X, r, \mu)$ are given the associated wavelet representation is typically reducible. By introducing wavelet filters associated to $(X, r)$ we build random walks in $X$, path-space measures, harmonic functions, and an associated Martin boundary. 

    We construct measures on the solenoid $(X_\infty, r_\infty)$, built from $(X, r)$. We show that $r_\infty$ induces unitary operators $U$ on Hilbert space $\H$ and  representations $\pi$  of the algebra $L^\infty(X)$ such that the pair $(U, r_\infty)$, together with the corresponding representation $\pi$ forms a crossed-product in the sense of $C^*$-algebras. We note that the traditional wavelet representations fall within this wider framework of $(\H, U, \pi)$ covariant crossed products. 
\end{abstract}
\maketitle \tableofcontents

\section{Introduction}

We study a decomposition problem for a class of unitary representations associated with wavelet analysis, even though our framework is wider and has applications outside multi-scale wavelet expansions, see details below.
    One powerful tool in the construction of families of multi-scale wavelets (see e.g., \cite{ Dau92, Ma98, Jor06} ) is an introduction of a finite system of filters. Here we understand the notion of ``filter'' in the sense of signal processing. In this context, each filter will be a function of a complex number  $z$ (a frequency variable), and for many purposes it is enough to consider only a phase of $z$, so we may restrict attention to the case when  points $z$ are in the $1$-torus. Each function $m_i(z)$, $i=0, 1,\dots, N-1$ typically supports a frequency band. With these conventions, in wavelet considerations, the function  $m_0$ represents a low-pass filter, i.e., passing low frequency signals. When a suitable Fourier expansion is introduced for the filter functions we arrive at the masking coefficients that determine some particular wavelet. This framework includes both traditional wavelet systems in the Hilbert space $L^2(\br^d)$ in some number of dimensions $d$, as well as orthonormal wavelet bases on fractals, as studied by two of the present authors, see \cite{DuJo07, DuJo06, BK10}.

     {\it Analysis of filters}.  Continuing with filter functions $m_0$ on the circle group $\bt$, we consider, for every fixed $z$ in $\bt$, then the absolute squared $m_0$ with some normalization $W(z) = |m_0( z )|^2/N$  . With this we then get a family of probability distributions on the set of $N$ solutions in $\bt$ to the equation $w^N = z$ (see \eqref{eqsi1}--\eqref{eqw2} below). There will be a solution $w$ in each of the $N$ frequency-bands.
In the special case when $N = 2$, the function $m_0$, or the system of functions $m_0$ and $m_1$, are called a quadrature mirror filter (QMF). The reason for this is that two other operations, down-sampling, and up-sampling, allow one to build a discrete wavelet algorithm with dual filters, the dual one is the ``mirror''.
 
   Here we have adopted a more general framework: Instead of $\bt$ we will consider a compact metric space $X$, and our filter functions will be functions from $X$ into the complex plane. We further generalize the choice of endomorphism $r(z)=  z^N$, considering here instead an endomorphism $r$ in $X$ which is onto, and for which each pre-image of points in $X$ is a finite subset; so finite-to-one endomorphisms.
In both the traditional wavelet case, and in the more general framework, we end up with dynamical systems in a solenoid.

    When the pair $(X, r)$ is given as specified, there is a standard way of building a solenoid $X_\infty = X_\infty(r)$  over $X$ (see \eqref{eqi4_1} and \eqref{eqi4_2} below).
    There are several advantages working with the solenoid:
    
    \begin{enumerate}
\item The endomorphism $r$ in $X$ induces an automorphism $r_\infty$ in $X_\infty$(see \eqref{eqi4_2}).
\item The transition probabilities $W$ on $X$ induce a random walk on the solenoid which encodes properties of the representations induced by the prescribed wavelet filters \cite{DuJo07}; these representations are known as {\it wavelet representations} in the literature.
\item With this random walk we are able to compute transition probabilities, harmonic functions, and associated Martin boundaries.
\item Points in the solenoid may be thought of as random-walk paths; for each point $x$ in $X$, we will have an infinite random-walk path, represented as a subset of the solenoid.
\item Fixing the function $W = |m_0(\cdot)|^2/\#r^{-1}(r(\cdot))$  on $X$ we get a path space measure $P$. We will be interested in the family of measures $(P_x)$, $x$ in $X$, with $P_x$ conditioned on the set of paths starting at $x$ (see \eqref{eqi4_5}).
 
\end{enumerate}
 
To specify wavelet representations we must then also have a prescribed measure $\mu$ on $X$ (see Definition \ref{def1.1}). For the theory of wavelet representations, see for example \cite{HL08, Lar07a, Lar07b, Pa04, Pa08a, Pa08b}. For an early treatment of transfer operators $R_W$ in multi-scale wavelets, see \cite{Jor01}. The relevance of the $R_W$ harmonic functions (i.e., solutions $h$ to $R_W h = h$) in the decomposition theory for the corresponding  wavelet representation $\rho_W$ was pointed out there. Specifically, \cite{Jor01} has the idea of computing operators in the commutant of $\rho_W$ from  $R_W$-harmonic functions.

For computations of $R_W$-harmonic functions of concrete wavelets, see also \cite{BrJo02}. The theory of $R_W$ harmonic functions was developed recently by a number of other authors; in \cite{Dut06} is introduced the study of intertwining operators for pairs of wavelet representations.

 In our present wider context, the $R_W$-harmonic functions enter in equation \eqref{eqw2} below, and they underlie our considerations in Section 4 regarding Martin boundary.

In our present setting, in an earlier paper \cite{DuSi10} a necessary condition was given on the data $(X, r, \mu)$ for when the associated wavelet representation is reducible. For the reader's convenience, we have stated it as Theorem \ref{thi5} below.
 
{\it Summary of results}. Our main result here (Theorem \ref{thdec}) offers a direct integral decomposition for the general wavelet representation.
    This completely solves a question posed by Professor Judith Packer, see e.g., \cite{Pa08b,MR2609543,MR2593341,MR2563776,MR2559716}. In our Theorem \ref{thdec} we offer a direct integral decomposition of the general wavelet representation, and our Theorem \ref{th1.2} deals with a derivation of the measures contributing to the decomposition.
    Our results yield as clean a decomposition for the general wavelet representation into irreducibles as is realistically feasible.

 In Section 4, we have included a result on an associated Martin boundary. Even though it is not used directly, it is certainly relevant for the understanding of our random-walk harmonic functions, and the question about the Martin boundary naturally presents itself. In fact, we obtained our result in Section 4 in response to a question asked of us by Erin Pearse.

     Indeed, there are some intriguing connections to random-walk models studied recently in papers by one of the present authors and E. Pearse (see e.g.,  \cite{JoPe10} and \cite{DuJo10}. These are computations for infinite weighted graphs $G$. In both instances we get transition probabilities and associated random walks.

     If $G$ is a graph as in \cite{JoPe10}, the condition for the context of these studies is that, for every vertex in $G$, there are only a finite number of transitions possible to neighboring vertices. But the Markov processes in \cite{JoPe10} are reversible, and therefore the associated boundaries are more amenable.
     By contrast, our transition processes are non-reversible, except for some trivial special cases. More specifically, the reason our random walks are not reversible, is that transition happens from one point $x$ to one of the distinct subnodes $y$ (neighbors) were $y$ will be one of the finite number of solutions to the equation  $r(y)=x$. Hence these transitions never return, unless we consider cases when $x$ might be a periodic point, in which case they might return.

\section{Measures on the solenoid}

Before turning to our direct integral result, we begin with some preliminaries regarding measures on solenoids. Since our starting point is a given finite-to-one endomorphism $r$ in a compact metric space $X$, it is then natural to look for a way of corresponding to this a unitary operator $U$ in a Hilbert space $\H$, such that $U$ together with $(X, r)$ satisfy a covariance relation; see (i) in Theorem \ref{thi1} below. The introduction of suitable measures on the associated solenoid $(X_\infty, r_\infty)$, built from $(X, r)$, then gets us a representation $\pi$  of the algebra $L^\infty(X)$ such that $U$, together with $r_\infty$, form a crossed-product in the sense of $C^*$-algebras. This is possible since $r_\infty$ is an automorphism.  We will refer to a crossed-product system $(\H, U, \pi)$ as a wavelet representation.

   Indeed, in \cite{DuJo07}, we proved that the traditional wavelet representations fall within this wider framework of $(\H, U, \pi)$ covariant crossed products.
Specifically, in the special case when $X = \bt$, and the endomorphism  $r$  is  just the power mapping  $r(z)=  z^N$  (for a fixed integer $N > 1$), then it can be seen that a covariant crossed products indeed specializes to a unitary representation of a corresponding $N$-Baumslag-Solitar group; see e.g., \cite{DuJo08c, Dut06}.  Even in the case of these classical Baumslag-Solitar groups, our understanding of the unitary representations and their decompositions is so far only partial.
 
\begin{definition}\label{def1.1}
Let $X$ be a compact metric space and $r:X\rightarrow X$ be a finite-to-one, onto, Borel measurable map. Let $\mu$ be a {\it strongly invariant} Borel probability measure on $X$, i.e.
\begin{equation}
\int f\,d\mu=\int\frac{1}{\#r^{-1}(x)}\sum_{r(y)=x}f(y)\,d\mu(x),
\label{eqsi1}
\end{equation}
for any bounded Borel function on $X$. 

A function $m_0$ on $X$ is called a {\it quadrature mirror filter (QMF)} if 
\begin{equation}
\frac{1}{\#r^{-1}(x)}\sum_{r(y)=x}|m_0(y)|^2=1,\quad(x\in X)
\label{eqqmf1}
\end{equation}

In what follows we will assume that:
\begin{equation}
\mbox{ the set of zeroes for $m_0$ has $\mu$-measure zero.}
\label{eqzer1}
\end{equation}

Given a QMF $m_0$ we define 
\begin{equation}
W(x)=\frac{|m_0(x)|^2}{\#r^{-1}(r(x))},\quad(x\in X).
\label{eqw1}
\end{equation}
Then the function $W$ satisfies the following equation:
\begin{equation}
\sum_{r(y)=x}W(y)=1,\quad(x\in X).
\label{eqw2}
\end{equation}
Equation \eqref{eqw2} can be interpreted as an assignment of transition probabilities: the probability of transition from $x$ to $y\in r^{-1}(x)$ is equal to $W(y)$. 

A function $h$ on $X$ is called {\it $R_W$-harmonic} if 
\begin{equation}
\sum_{r(y)=x}W(y)h(y)=h(x),\quad(x\in X).
\label{eqrwh}
\end{equation}
\end{definition}

\begin{theorem}\label{thi1}\cite{DuJo07}
There exists a Hilbert space $\H$, a unitary operator $U$ on $\H$, a representation $\pi$ of $L^\infty(X)$ on $\H$ and an element $\varphi$ of $\H$ such that
\begin{enumerate}
\item 
(Covariance) $U\pi(f)U^{-1}=\pi(f\circ r)$ for all $f\in L^\infty(X)$. 
\item (Scaling equation) $U\varphi=\pi(m_0)\varphi$
\item (Orthogonality) $\ip{\pi(f)\varphi}{\varphi}=\int f\,d\mu$ for all $f\in L^\infty(X)$. 
\item (Density) $\{U^{-n}\pi(f)\varphi\,|\,n\in\bn, f\in L^\infty(X)\}$ is dense in $\H$.

\end{enumerate}
Moreover they are unique up to isomorphism.

\end{theorem}

\begin{definition}
We call the system $(\H,U,\pi,varphi)$ in Theorem \ref{th1}, the {\it wavelet representation} associated to the function $m_0$.

\end{definition}

We recall some facts from \cite{DuJo07}. The wavelet representation can be realized on a solenoid as follows.
Let 
\begin{equation}
X_\infty:=\left\{(x_0,x_1,\dots)\in X^{\bn}\,|\, r(x_{n+1})=x_n\mbox{ for all }n\geq 0\right\}.
\label{eqi4_1}
\end{equation}
We call $X_\infty$ the {\it solenoid} associated to the map $r$.

On $X_\infty$ consider the $\sigma$-algebra generated by cylinder sets. 
Let $r_\infty:X_\infty\rightarrow X_\infty$
\begin{equation}
r_\infty(x_0,x_1,\dots)=(r(x_0),x_0,x_1,\dots)\mbox{ for all }(x_0,x_1,\dots)\in X_\infty.
\label{eqi4_2}
\end{equation}
Then $r_\infty$ is a measurable automorphism on $X_\infty$. 

Define $\theta_0:X_\infty\rightarrow X$, 
\begin{equation}
\theta_0(x_0,x_1,\dots)=x_0.
\label{eqi4_3}
\end{equation}

The following commutative diagram summarizes the relation between the maps $r,r_\infty,\theta_0$:
$$\begin{array}{rcl}
X_\infty&\stackrel{r_\infty}{\rightarrow}&X_\infty\\
\theta_0\downarrow& &\downarrow\theta_0\\
X&\stackrel{r}{\rightarrow}&X\end{array},\quad \theta_0\circ r_\infty=r\circ\theta_0$$

Define for $m\geq0$ the projection $\theta_m:X_\infty\rightarrow X$, 
$$\theta_m(x_0,x_1,\dots)=x_m.$$

The measure $\mu_\infty$ on $X_\infty$ will be defined by constructing some path measures $P_x$ on the fibers $\Omega_x:=\{(x_0,x_1,\dots)\in X_\infty\,|\, x_0=x\}.$

On $\Omega_{x_0}$ we will consider the infinite product topology which is defined by the basis of open sets: for $n\geq 0$, $x_1,\dots,x_n\in X$ with $r(x_{j+1})=x_{j}$, $j\in\{0,\dots,n-1\}$,

\begin{equation}
V_{x_0,\dots, x_n}:=\{(z_0,z_1,\dots)\in\Omega_{x_0} : z_0=x_0,\dots, z_n=x_n\}.
\label{eqvx0}
\end{equation}
With this topology $\Omega_{x_0}$ is a compact Hausdorff space.

Let 
$$c(x):=\#r^{-1}(r(x)),\quad W(x)=|m_0(x)|^2/c(x),\quad(x\in X).$$
Then 
\begin{equation}
\sum_{r(y)=x}W(y)=1,\quad(x\in X)
\label{eqi4_4}
\end{equation}
$W(y)$ can be thought of as the trasition probability from $x=r(y)$ to one of its roots $y$.

For $x\in X$, the path measure $P_x$ on $\Omega_x$ is defined on cylinder sets by
\begin{equation}
P_x(\{(x_n)_{n\geq0}\in\Omega_x\,|\, x_1=z_1,\dots,x_n=z_n\})=W(z_1)\dots W(z_n)
\label{eqi4_5}
\end{equation}
for any $z_1,\dots, z_n\in X$.

This value can be interpreted as the probability of the random walk to go from $x$ to $z_n$ through the points $x_1,\dots,x_n$.

Next, define the measure $\mu_\infty$ on $X_\infty$ by 
\begin{equation}
\int f\,d\mu_\infty=\int_{X}\int_{\Omega_x}f(x,x_1,\dots)\,dP_{x}(x,x_1,\dots)\,d\mu(x)
\label{eqi4_7}
\end{equation}
for bounded measurable functions on $X_\infty$.

Consider now the Hilbert space $\H=L^2(\mu_\infty)$. Define the operator 
\begin{equation}
U\xi=(m_0\circ\theta_0)\,\xi\circ r_\infty,\quad(\xi\in L^2(X_\infty,\mu_\infty)).
\label{eqi4_8}
\end{equation}

Define the representation of $L^\infty(X)$ on $\H$ 
\begin{equation}
\pi(f)\xi=(f\circ\theta_0)\,\xi,\quad(f\in L^\infty(X),\xi\in L^2(X_\infty,\mu_\infty)).
\label{eqi4_9}
\end{equation}

Let $\varphi=1$ be the constant function $1$ on $X_\infty$.

\begin{theorem}\label{thi5}\cite{DuJo07}
Suppose $m_0$ is non-singular, i.e., $\mu(\{x\in X\,|\, m_0(x)=0\})=0$. Then the data $(\H,U,\pi,\varphi)$ forms the wavelet representation associated to $m_0$.
\end{theorem}

\begin{theorem}\label{th1}\cite{DuSi10}
Suppose $r:(X,\mu)\rightarrow (X,\mu)$ is ergodic. Assume $|m_0|$ is not constant $1$ $\mu$-a.e., non-singular, i.e., $\mu(m_0(x)=0)=0$, and $\log |m_0|^2$ is in $L^1(X)$. Then the wavelet representation $(\H,U,\pi,\varphi)$ is reducible. 
\end{theorem}

We will be interested in the decomposition of the wavelet representation into irreducibles. We need a few more notations and lemmas.
\begin{definition}\label{def2}
Define 
$$\tilde m_0=1,\quad \tilde m_n=(m_0\circ\theta_0)\cdot (m_0\circ\theta_0\circ\ r_\infty)\dots(m_0\circ\theta_0\circ r_\infty^{n-1}),\mbox{ for }n\geq 1,$$
$$\tilde m_n=\frac{1}{(m_0\circ\theta_0\circ\ r_\infty^{-1})\dots (m_0\circ\theta_0\circ r_\infty^{n})},\mbox{ for }n<0.$$

The function $\tilde m: X_{\infty}\times \Z\to \C^{\ast}$ defined by $\tilde m(x,n)=\tilde m_n(x)$ gives a one-cocycle for the action of $\Z$ on $X_{\infty}$ determined by $r_{\infty}.$
\end{definition}

The fact that $U$ is an isometry implies the following lemma.
\begin{lemma}\label{lem3}
For $\xi\in L^2(X_\infty,\mu_\infty)$:
$$\int \xi\,d\mu_\infty=\int |\tilde m_n|^2\xi\circ r_\infty^n\,d\mu_\infty,\quad(n\in\bz).$$
\end{lemma}

\section{The decomposition of the wavelet representation}

\begin{definition}\label{deffu}
We say that a subset $\F$ of $X_\infty$ is a {\it fundamental domain} if, up to $\mu_\infty$-measure zero:
$$\bigcup_{n\in\bz} r_\infty^n(\F)=X_\infty\quad\mbox{ and }\quad r_\infty^n(\F)\cap r_\infty^m(\F)=\ty\mbox{ for }n\neq m.$$
\end{definition}

\begin{definition}\label{defir}
For $z=(z_0,z_1,\dots)$ in $X_\infty$ define the following representation: consider the Hilbert space 
$$\H_z:=\left\{(\xi_n)_{n\in\bz} : \sum_{n\in\bz} |\xi_n|^2|\tilde m_n(z)|^2<\infty\right\},$$
with inner product
$$\ip{\xi}{\eta}_{\H_z}:=\sum_{n\in\bz}\xi_n\cj\eta_n |\tilde m_n(z)|^2.$$

Note that we avoid here the points $z\in X_\infty$ such that one of the functions $\tilde m_n(z)=0$. Since $m_0$ is non-singular, such points form a set of $\mu_\infty$-measure zero. 

Define the unitary operator 
$$U_z(\xi_n)_{n\in\bz}=(m_0\circ\theta_0\circ r_\infty^n(z)\xi_{n+1})_{n\in\bz}.$$
Define the representation of $\pi$ of $L^\infty(X)$: 
$$\pi_z(f)(\xi_n)_{n\in\bz}=(f\circ\theta_0\circ r_\infty^n(z)\xi_n)_{n\in\bz},\quad(f\in L^\infty(X)).$$

The representation $\pi_z$ is defined for bounded functions on $X$, not just essentially bounded. The $\mu$-measure zero sets will affect the individual representations $\pi_z$ but not their direct integral (see below). 
\end{definition}

\begin{theorem}\label{thdec}
In the hypotheses of Theorem \ref{th1}, there exist a fundamental domain $\F$. The wavelet representation associated to $m_0$ has the following direct integral decomposition:
$$[\H,U,\pi]=\int_{\F}^\oplus[\H_z,U_z,\pi_z]\,d\mu_\infty(z),$$
 where the component representations $[\H_z,U_z,\pi_z]$ in the decomposition are irreducible for a.e., $z$ in $\F$, relative to $\mu_\infty$.
\end{theorem}
\begin{proof}
We state the irreducibility of the component representations in a lemma:

\begin{lemma}\label{pr4}
For $\mu_\infty$ almost every $z\in X_\infty$, the objects $[\H_z,U_z,\pi_z]$ form an irreducible representation. 
\end{lemma}

\begin{proof}
One has to check that $U_z$ is unitary, $\pi_z$ is a representation and $U_z\pi_z(f)U_z^{-1}=\pi_z(f\circ r)$ for all $f\in L^\infty(X)$. All these follow from simple computations. 

To see that the representation is irreducible for $\mu_\infty$-a.e. $z$, take $z$ to be non-periodic, i.e., $r_\infty^n(z)\neq z$ for all $n\neq 0$. Then $\{\pi_z(f) : f\in L^\infty(X)\}$ forms a maximal abelian subalgebra with cyclic vector $\delta_0$ (see \cite[Corollary III.1.3]{Tak02}), where $\delta_0(n)=1$ for $n=0$, and $\delta_0(n)=0$ otherwise. Then, an operator $A$ that commutes with $U_z$ and $\pi_z$ has to be of the form $\pi_z(g)$ for some $g\in L^\infty(X)$. Since $A$ commutes with $U_z$ we have $\pi_z(g\circ r)=U_z\pi_z(g)U_z^{-1}=\pi_z(g)$. This implies that $g$ is constant on $\{r_\infty^n(z) : n\in\bz\}$, so $A$ is a multiple of the identity. 
\end{proof}

We begin the proof as in the proof of the main result in \cite{DuSi10}.

From the QMF relation and the strong invariance of $\mu$ we have
$$\int_X|m_0|^2\,d\mu=\int_X\frac{1}{\# r^{-1}(x)}\sum_{r(y)=x}|m_0(y)|^2\,d\mu=1.$$

By Jensen's inequality we have 

$$a:=\int_X \log |m_0|^2\,d\mu\leq \log\int_X|m_0|^2\,d\mu=0.$$
Since $\log$ is strictly concave, and $|m_0|^2$ is not constant $\mu$-a.e., it follows that the inequality is strict, and $a<0$.

Since $r$ is ergodic, applying Birkoff's ergodic theorem, we obtain that 
$$\lim_{n\rightarrow\infty}\frac{1}{n}\sum_{k=0}^{n-1}\log|m_0\circ r^k|^2=\int_X\log|m_0|^2\,d\mu=a,\,\mu-\mbox{ a.e.}$$
This implies that 
$$\lim_{n\rightarrow\infty}\left(|m_0(x)m_0(r(x))\dots m_0(r^{n-1}(x))|^2\right)^{1/n}=e^a<1,\,\mu-\mbox{a.e.}$$
Take $b$ with $e^a<b<1$.

By Egorov's theorem, there exists a measurable set $A_0$, with $\mu(A_0)>0$, such that $$(|m_0(x)m_0(r(x))\dots m_0(r^{n-1}(x))|^2)^{1/n}$$ converges uniformly to $e^a$ on $A_0$. (Taking $A_0$ smaller if needed we can assume $\mu(A_0)<1$.) This implies that there exists an $n_0$ such that for all $m\geq n_0$:
$$\left(|m_0(x)m_0(r(x))\dots m_0(r^{m-1}(x))|^2\right)^{1/m}\leq b\mbox{ for }x\in A_0$$
and so
\begin{equation}
|m_0(x)m_0(r(x))\dots m_0(r^{m-1}(x))|^2\leq b^m,\mbox{ for $m\geq n_0$ and all $x\in A_0$.}
\label{eqt1}
\end{equation}

Next, given $m\in\bn$, we compute the probability of a sequence $(z_n)_{n\in\bn}\in X_\infty$ to have $z_m\in A_0$. We have, using the strong invariance of $\mu$:
$$P(z_m\in A_0)=\mu_\infty\left(\{(z_n)_n\,|\, z_m\in A_0\}\right)=\int_{X_\infty}\chi_{A_0}\circ\theta_m\,d\mu_\infty$$
$$=\int_X\frac{1}{\#r^{-m}(z_0)}\sum_{r(z_1)=z_0,\dots, r(z_m)=z_{m-1}}|m_0(z_1)|^2\dots |m_0(z_m)|^2\chi_{A_0}(z_m)\,d\mu(z_0)$$
$$=\int_X|m_0(z_m)m_0(r(z_m))\dots m_0(r^{m-1}(z_m))|^2\chi_{A_0}(z_m)\,d\mu(z_m)$$
$$=\int_X|m_0(x)m_0(r(x))\dots m_0(r^{m-1}(x))|^2\chi_{A_0}(x)\,d\mu(x).$$

Then 
$$\sum_{m=1}^\infty P(z_m\in A_0)=\sum_{m\geq1}\int_X |m_0(x)m_0(r(x))\dots m_0(r^{m-1}(x))|^2\chi_{A_0}\,d\mu(x)<\infty$$
and we used \eqref{eqt1} in the last inequality.

Now we can use Borel-Cantelli's lemma, to conclude that the probability that $z_m\in A_0$ infinitely often is zero. Thus, for $\mu_\infty$-a.e. $z:=(z_n)_n$, there exists $k_z$ (depending on the point) such that $z_n\not\in A_0$ for $n\geq k_z$. In other words, if $B_0=X\setminus A_0$ then for $\mu_\infty$-a.e. $(z_n)_n$ in $X_\infty$ there exists $k_z$ such that $z_n\in B_0$ for all $n\geq k_z$.

Define now the set 
$$A_\infty:=\left\{(z_0,z_1,\dots)\in X_\infty : z_0,z_1,\dots\in B_0\right\}.$$
It is clear that if $(z_0,z_1,\dots)\in A_\infty$ then $r_\infty^{-1}(z_0,z_1,\dots)=(z_1,z_2,\dots)$ is in $A_\infty$ too. Therefore $r_\infty^{-1}(A_\infty)\subseteq A_\infty$. This means also that $A_\infty\subseteq r_\infty(A_\infty)$.

From the statements above we see that for $\mu_\infty$-a.e. $(z_0,z_1,\dots)$ in $X_\infty$ there exists $n$ such that $z_n,z_{n+1},\dots$ are in $B_0$ which means that $r_\infty^{-n}(z_0,z_1,\dots)$ is in $A_\infty$ and so $(z_0,z_1,\dots)\in r_\infty^n(A_\infty).$ Thus, up to measure zero:
$$\bigcup_{n\in\bz} r_\infty^n(A_\infty)=X_\infty.$$

We claim that also, up to measure zero, one has 
$$\bigcap_{n\in\bz}r_\infty^n(A_\infty)=\ty.$$

Suppose $(z_0,z_1,\dots)$ is in all $r_\infty^{-n}(A_\infty)$ for $n\geq 0$. Then $(r^n(z_0),r^{n-1}(z_0),\dots)\in A_\infty$ so $r^n(z_0)\in B_0$ for all $n\geq0$. Since $0<\mu(B_0)<1$ this contradicts the fact that $r$ is ergodic on $X$. 

Now take $\F:=r_\infty(A_\infty)\setminus A_\infty$. The properties of $A_\infty$ easily imply that $\F$ is a fundamental domain.

Next we check the direct integral decomposition. 

Define $\Psi:L^2(X_\infty,\mu_\infty)\rightarrow \int_{\F}^\oplus\H_z\,d\mu_\infty(z),$
$$(\Psi\xi)(z)=(\xi\circ r_\infty^n(z))_{n\in\bz},\quad(\xi\in L^2(X_\infty,\mu_\infty),z\in\F).$$
We check that $\Psi$ is an isometry. We use Lemma \ref{lem3}:
$$\|\xi\|^2=\sum_{n\in\bz}\int|\xi|^2\chi_{r_\infty^n(\F)}\,d\mu_\infty=\sum_{n\in\bz}\int|\tilde m_n|^2|\xi\circ r_\infty^n|^2\chi_{\F}\,d\mu_\infty=\int_{\F}\sum_{n\in\bz}|\tilde m_n|^2|\xi\circ r_\infty^n|^2\,d\mu_\infty$$
$$=\int_{\F}\|(\Psi\xi)(z)\|_{\H_z}^2\,d\mu_\infty(z)=\|\Psi\xi\|^2.$$

To check that $\Psi$ is onto, we can compute the inverse 
$(\Psi^{-1}(\xi(\cdot)_n)_{n\in\bz})(z)=\xi_n(r_\infty^{-n}z)$ if $z\in r_\infty^n(\F)$.

Some direct computations show that $\Psi$ intertwines the $U$-operators and the representations $\pi$.

\end{proof}

\section{Martin boundary}

The idea of associating to wavelet constructions a transfer operator $R_W$ and associated harmonic functions was pioneered by W. Lawton  in the two papers  \cite{Law90,Law91} .

The idea is that wavelets are determined by a system of numbers, often called masking coefficients. It is possible to turn these into coefficients in filter functions $m_i$, and by selecting $i=0$ (see eq \eqref{eqw2}) we get transition probabilities and a transfer operator $R_W$,

$$R_Wf(x)=\sum_{r(y)=x}W(y)f(y),\quad(x\in X).$$

     Hence $R_W$ is determined by the prescribed masking coefficients, and the question is how properties of the masking coefficients (and therefore of $R_W$) determine the wavelets. It turns out that this is decided by the spectrum of $R_W$, including the eigenspace for eigenvalue 1, which produces the harmonic functions.

As shown in \cite{DLS10}, operators in the commutant of the wavelet representation correspond to bounded $R_W$-harmonic functions. If we restrict such harmonic functions to inverse orbits of points we get harmonic functions for the random walk, or what we call below $p$-harmonic functions. The Martin boundary theory offers a way to construct such harmonic functions by means of integrals on a certain boundary. We perform these computations here to see what the $p$-harmonic functions are in this case.  
\begin{definition}\label{def1.2}
A point $x_0\in X$ is called {\it regular} if the following two conditions are satisfied:
\begin{enumerate}
	\item The sets $r^{-n}(x_0)$, $n\in\bn$ are mutually disjoint. 
	\item None of the sets $r^{-n}(x_0)$, $n\geq0$ intersect the set of zeroes of $W$.  
\end{enumerate}
Note that condition (i) means that $x_0$ is not periodic for the map $r$, i.e., $r^n(x_0)\neq x_0$ for any $n\geq1$. 

For a point $x_0\in X$, define the set $\T(x_0):=\cup_{n\geq 0}r^{-n}(x_0)$. We call this the {\it tree with root at $x_0$}. 
If $x_0$ is regular and $x\in \T(x_0)$, define $n(x_0)$ to be the unique non-negative integer such that $r^{n(x_0)}(x)=x_0$.

Let $x_0\in X$ be regular. We define now a random walk on the set $\T(x_0)$ and we construct its Martin boundary by following \cite{Saw}. 

For $x,y\in \T(x_0)$ define the transition probabilities $p(x,y)$ as follows:
\begin{equation}
p(x,y):=\left\{\begin{array}{cc}
W(y),&\mbox{ if } r(y)=x,\\
0,&\mbox{otherwise.}\end{array}\right.
\label{eqpxy}
\end{equation}

A function $u$ on $\T(x_0)$ is called {\it $p$-harmonic} if
\begin{equation}
u(x)=\sum_{y\in\T(x_0)}p(x,y)u(y),\quad(x\in\T(x_0)).
\label{eqpha}
\end{equation}

The function $p_n(x,y)$ is the $n$-th matrix power of $p(x,y)$ and represents the probability of transition from $x$ to $y$ in $n$ steps. It can be easily seen that $p_0(x,y)=\delta_{xy}$ and 
\begin{equation}
p_n(x,y)=\left\{\begin{array}{cc}
W(y)W(r(y))\dots W(r^{n-1}(y)),&\mbox{ if } r^n(y)=x,\\
0,&\mbox{otherwise.}\end{array}\right.
\label{eqpnxy}
\end{equation}

The {\it Green function} or {\it potential function } is defined by 
\begin{equation}
g(x,y):=\sum_{n=0}^\infty p_n(x,y),\quad(x,y\in \T(x_0)).
\label{eqggen}
\end{equation}
Note that, for our random walk, only one term in the sum in \eqref{eqggen} is non-zero.

The {\it Martin kernel} is defined by
\begin{equation}
K(x,y):=\frac{g(x,y)}{g(x_0,y)},\quad(x,y\in\T(x_0)).
\label{eqkgen}
\end{equation}
The denominator in \eqref{eqkgen} is non-zero because each vertex $y$ can be reached from $x_0$ eventually.

The function $K(x,\cdot)$ is bounded by some constant $C_x$ (which we will describe below). Set 
\begin{equation}
\rho(x,y)=\sum_{q\in \T(x_0)}D(q)\frac{|K(q,x)-K(q,y)|+|\delta_{qx}-\delta_{qy}|}{C_q+1},\quad(x,y\in \T(x_0)),
\label{eqrho}
\end{equation}
where $D(q)>0$ for all $q\in \T(x_0)$ and $\sum_{q\in \T(x_0)}D(q)<\infty$. Here $\delta_{xy}=1$ if $x=y$ and $\delta_{xy}=0$ otherwise.

The {\it Martin compactification} $[\widehat \T(x_0),\widehat\rho]$ is the completion of $\T(x_0)$ with the metric $\rho$. The {\it Martin boundary} is defined as $\partial \T(x_0):=\widehat\T(x_0)\setminus\T(x_0)$.

 As shown in \cite{Saw} a sequence $\{y_n\}$ in $\T(x_0)$ is Cauchy with respect to the metric $\rho$ if and only if either (i) $y_n=y$ for all $n\geq n_0$ for some $y\in \T(x_0)$ and some $n_0\in\bn$, or else (ii) $\lim_{n\rightarrow\infty}y_n=\infty$ and $\lim_{n\rightarrow \infty}K(x,y_n)$ exists for all $x\in \T(x_0)$. (Here $\lim_{n\rightarrow\infty} y_n=\infty$ means that $y_n$ leaves eventually any finite set and never returns.)

Thus the Martin boundary $\partial\T(x_0)$ is the set of equivalence classes of Cauchy sequences that satisfy the condition (ii) above. 

The maps $K(x,\cdot)$, $x\in \T(x_0)$ extend uniquely to continuous maps on $\widehat\T(x_0)$ and we use the same notation $K(x,\cdot)$ for their extensions. 
\end{definition}

\begin{theorem}\label{thmrt}
{\bf [Martin representation theorem]} For any $p$-harmonic function $u(x)\geq 0$ there exists a measure $\nu$ on $\partial\T(x_0)$ such that 
\begin{equation}
u(x)=\int_{\partial\T(x_0)}K(x,\alpha)\,d\nu(\alpha),\quad(x\in\T(x_0)).
\label{eqmrt}
\end{equation}
\end{theorem}

\begin{proposition}\label{pr1.3}
With the definitions above we have:
\begin{enumerate}
	\item The Green function satisfies the equation
	\begin{equation}
g(x,y)=\left\{\begin{array}{cc}
W(y)W(r(y))\dots W(r^{n-1}(y)),&\mbox{ if } r^n(y)=x\mbox{ for some $n\geq 0$},\\
0,&\mbox{otherwise.}\end{array}\right.
\label{eqgxy}
\end{equation}
(If $n=0$ the product is defined to be 1.)
\item The Martin kernel is 
\begin{equation}
K(x,y)=\left\{\begin{array}{cc}
\frac{1}{W(x)W(r(x))\dots W(r^{n(x)-1}(x))}=:C_x,&\mbox{ if } r^n(y)=x\mbox{ for some $n\geq0$},\\
0,&\mbox{otherwise.}\end{array}\right.
\label{eqkxy}
\end{equation}
Thus $K(x,\cdot)$ is constant $C_x$ on the subtree $\T(x)$ with root at $x$ and $0$ everywhere else.
Using the notation of Definition \ref{def2} we have that $K(x,y)=\frac{1}{\tilde W_n(x)}$, if $r^n(y)=x$; if $\#r^{-1}(x)$ is constant, then the two functions $\tilde W_n$ and $\tilde m_n$ differ by a multiplicative constant (see \eqref{eqw1}).
\item A function $u$ on $\T(x_0)$ is $p$-harmonic if and only if 
\begin{equation}
u(x)=\sum_{r(y)=x}W(y)u(y),\quad(x\in\T(x_0)).
\label{eqwha}
\end{equation}
\end{enumerate}

\end{proposition}

\begin{proof}
(i) follows directly from \eqref{eqpnxy}. Note that, because $x_0$ is regular, the number $n$ such that $r^n(y)=x$ is unique. For (ii), if $r^n(y)\neq x$ for all $n$ then $g(x,y)=0$ so $K(x,y)=0$. If $r^n(y)=x$, then we have
$$K(x,y)=\frac{g(x,y)}{g(x_0,y)}=\frac{W(y)\dots W(r^{n-1}(y))}{W(y)\dots W(r^{n(y)-1}(y))}=\frac{1}{W(r^n(y))\dots W(r^{n(y)-1}(y))}$$$$=\frac{1}{W(x)\dots W(r^{n(x)-1}(x))}.$$

(iii) is obtained from the following computation:
$$u(x)=\sum_y p(x,y)u(y)=\sum_{r(y)=x}p(x,y)u(y)=\sum_{r(y)=x}W(y)u(y).$$
\end{proof}

\begin{theorem}\label{th1.1}
Let $x_0$ be a regular point in $X$. Define the map $\Phi:\Omega_{x_0}\rightarrow\partial \T(x_0)$ by 
\begin{equation}
\Phi(x_0,x_1,\dots):=\{x_n\},
\label{eq1.1.1}
\end{equation}
i.e., to each sequence in $\Omega_{x_0}$ we associate the equivalence class of this sequence in $\partial\T(x_0)$. 

Then $\Phi$ is a bijective homeomorphism from $\Omega_{x_0}$ onto $\partial\T(x_0)$.

For $x\in\T(x_0)$ and $(x_0,x_1,\dots)\in\Omega_{x_0}$
\begin{equation}
K(x,\Phi(x_0,x_1,\dots))=\left\{\begin{array}{cc}
\frac{1}{W(x_1)W(x_2)\dots W(x_n)},&\mbox{ if $x=x_n$ for some $n\geq0$},\\
0,&\mbox{ otherwise.}
\end{array}\right.
\label{eqkxb}
\end{equation}
\end{theorem}

\begin{proof}
First, we show that $\Phi$ is well defined, so $\{x_n\}$ is a sequence with the property that $\lim x_n=\infty$ and $\lim K(x,x_n)$ exists for all $x\in\T(x_0)$. 

Since $x_0$ is regular, the sets $r^{n}(x_0)$ are disjoint, therefore any finite subset of $\T(x_0)$ lies in a finite union $\cup_{j\leq J}r^{-j}(x_0)$, and since $x_n\in r^{-n}(x_0)$ for all $n$, it follows that $x_n$ eventually leaves this finite set and never returns. 

For the second condition, take $x\in \T(x_0)$. Recall that $n(x)$ is the unique number such that $x\in r^{-n(x)}(x_0)$. We have two possibilities: $x_{n(x)}=x$ or not. 
In the first case we have that $x_n$ is in the subtree $\T(x)$ for all $n\geq n(x)$ so $K(x,x_n)$ is constant $C_x$, by Proposition \ref{pr1.3}. In the second case, we have that $x_n$ is not in the subtree $\T(x)$ for all $n\geq n(x)$, so $K(x,x_n)$ is constant 0. In both cases $\lim K(x,x_n)$ exists, so $\Phi$ is well defined. 

Next, we check that $\Phi$ is onto. Take a sequence $\{y_n\}$ in $\T(x_0)$ with $\lim y_n=\infty$ and such that $\lim K(x,y_n)$ exists for all $x\in\T(x_0)$.
Since $K(x,y)$ is either $C_x>0$ or $0$ depending on whether $y$ is in the subtree $\T(x)$ or not, it follows that for all $x\in\T(x_0)$, either the sequence is $y_n$ is eventually contained in the subtree $\T(x)$ or it is eventually contained in the complement of $\T(x)$; it cannot jump back and forth between $\T(x)$ and its complement. 

We will construct by induction the sequence $(x_n)_n$ in $\Omega_{x_0}$ with $\Phi(x_0,x_1,\dots)=\{y_n\}.$
The first element $x_0$ is given. Next consider the points $z_1,\dots,z_J$ in $r^{-1}(x_0)$. Since $\lim y_n=\infty$, eventually the sequence will be in the union of the subtrees $\cup_{j=1}^J \T(z_j)$. With the previous remark, one of the sets $\T(z_j)$ will contain the entire sequence $y_n$ eventually. We define $x_1$ to be the point $z_j$ with this property. 

Inductively, if $x_m$ has been defined such that the entire sequence $y_n$ is in the subtree $\T(x_m)$ eventually, we take the points in $r^{-1}(x_m)$; since $\lim y_n=\infty$, the entire sequence will lie in $\cup_{z\in r^{-1}(x_m)}\T(z)$ eventually. Since $y_n$ cannot jump back and forth between a subtree and its complement, there is one of the elements $z\in r^{-1}(x_m)$ such that the sequence $y_n$ lies in the subtree $\T(z)$ eventually. We call this point $x_{m+1}$.

To prove that $\Phi(x_0,x_1\dots)=\{y_n\}$ we just have to show that the sequences $\{x_n\}$ and $\{y_n\}$ are equivalent, i.e., $\lim \rho(x_n,y_n)=0$.
Take $\epsilon>0$. There exists a finite subset $F$ of $\T(x_0)$ such that $\sum_{q\not\in F}2D(q)<\epsilon$. For each $q\in F$, either $y_n$ is in the subtree $\T(y_n)$ eventually, or it is in the complement of $\T(y_n)$ eventually. From the definition of $\{x_n\}$ we see that $x_n$ will have exactly the same property. Thus $K(q,x_n)=K(q,y_n)$ for $n$ large enough. Also, since $x_n$ and $y_n$ go to infinity, it follows that $\delta_{qx_n}=\delta_{qy_n}$ for $n$ large enough. 
Therefore, for $n$ large, the terms in the sum in \eqref{eqrho} for $\rho(x_n,y_n)$ that correspond to $q\in F$ are all zero, the rest are bounded by $2\sum_{q\not\in F}D(q)<\epsilon$. So $\rho(x_n,y_n)<\epsilon$ for $n$ large, and therefore $\Phi(x_0,x_1,\dots)=\{y_n\}$ and $\Phi$ is onto.

To see that $\Phi$ is one-to-one, take $(x_n)\neq (x_n')$ in $\Omega_{x_0}$. Let $n_0\geq 1$ such that $x_{n_0}\neq x_{n_0}'$. Then for $n\geq n_0$,
$K(x_{n_0},x_n)=C_{x_{n_0}}$ and $K(x_{n_0},x_{n}')=0$. Therefore 
$$\rho(x_n,x_n')\geq \frac{|K(x_{n_0},x_n)-K(x_{n_0},x_n')|}{C_{x_{n_0}}+1}=\frac{C_{x_{n_0}}}{C_{x_{n_0}}+1}.$$
This implies that $\Phi(x_n)\neq \Phi(x_n')$ so $\Phi$ is one-to-one.

To prove that $\Phi$ is continuous, take $(x_n)\in\Omega_{x_0}$ and $\epsilon>0$. Take a finite subset $F$ of $\T(x_0)$ such that $2\sum_{q\not\in F}D(q)<\epsilon$. Take $n_0$ such that $F$ is contained in $\cup_{n\leq n_0}r^{-n}(x_0)$. Take $(x_n')\in V_{x_0,\dots,x_{n_0}}$ so 
$x_0'=x_0,\dots, x_{n_0}'=x_{n_0}$. Then the sequences $x_n$ and $x_n'$ are in the subtree $\T(x_{n_0})$ for $n\geq n_0$. This implies that for $q\in F$ and $n>n_0$, we have that either both $x_n$ and $x_n'$ lies in the subtree $\T(q)$ or they both lie outside $\T(q)$; also $\delta_{qx_n}=\delta_{qx_n'}=0$. Then 
$\rho(x_n,x_n')\leq \sum_{q\not\in F}2D(q)<\epsilon$ so $\widehat\rho(\Phi(x_n),\Phi(x_n'))\leq\epsilon$. This shows that $\Phi$ is continuous.

Since both spaces $\Omega_{x_0}$ and $\partial \T(x_0)$ are compact Hausdorff, it follows that $\Phi$ is a homeomorphism.

Next, we check \eqref{eqkxb}. Take $x\in\T(x_0)$ and $(x_0,x_1,\dots)\in\Omega_{x_0}$. We have 
$K(x,\Phi(x_0,x_1,\dots))=\lim K(x,x_n)$. If $x\neq x_n$ for all $n\geq0$, then $x_n$ is not in the subtree $\T(x)$ for all $n$, so $K(x,x_n)=0$ for all $n$ and therefore $K(x,\Phi(x_0,x_1,\dots))=0$. 

If $x=x_n$ for some $n$, then for $m\geq n$, $x_m$ is in the subtree $\T(x)$ so 
$$K(x,x_m)=K(x_n,x_m)=\frac{1}{W(x_n)W(r(x_n))\dots W(r^{n(x_n)-1}(x_n))}=\frac{1}{W(x_n)W(x_{n-1})\dots W(x_1)}.$$
This proves \eqref{eqkxb}.
\end{proof}

\begin{theorem}\label{th1.2}
For any $p$-harmonic function $u\geq 0$, there exists a measure $\nu$ on $\Omega_{x_0}$ such that 
\begin{equation}
u(x)=\frac{1}{W(x)W(r(x))\dots W(r^{n(x)-1}(x))}\nu(V_{r^{n(x)}(x),r^{n(x)-1}(x),\dots,x}),\quad(x\in \T(x_0)).
\label{eqbrep}
\end{equation}

\end{theorem}

\begin{proof}
By Theorem \ref{thmrt}, there exists a measure $\hat\nu$ on $\partial \T(x_0)$ such that 
$$u(x)=\int_{\partial\T(x_0)}K(x,\alpha)\,d\hat\nu(\alpha),\quad(x\in\T(x_0)).$$
Define the measure $\nu$ on $\Omega_{x_0}$ by $\nu=\hat\nu\circ\Phi$. 
Then 
$$u(x)=\int_{\Omega_{x_0}}K(x,\Phi(x_0,x_1,\dots))\,d\nu(x_0,x_1,\dots).$$
But $K(x,\Phi(\cdot))$ is a multiple $C_x$ of the characteristic function of $V_{r^{n(x)}(x),r^{n(x)-1}(x),\dots, x}$. From this, \eqref{eqbrep} follows immediately.
\end{proof}

\begin{definition}\label{defadd}
A non-negative function $\nu$ on $\T(x_0)$ is called {\it additive} if 
\begin{equation}
\nu(x)=\sum_{r(y)=x}\nu(y),\quad(x\in\T(x_0)).
\label{eqadd}
\end{equation}

Denote by 
$$W^{(n)}(x)=W(x)\dots W(r^{n-1}(x)).$$

\end{definition}

\begin{corollary}\label{cor1.2}
For any $p$-harmonic function $u\geq 0$ there exists a unique additive function $\nu$ such that 
\begin{equation}
u(x)=\frac{1}{W^{(n(x))}(x)}\nu(x),\quad(x\in\T(x_0)).
\label{eqco1.2}
\end{equation}
Conversely, if $u$ is given by \eqref{eqco1.2} for some additive function $\nu$, then $u$ is $p$-harmonic. 
\end{corollary}
\begin{proof}
The existence can be obtained from Theorem \ref{th1.2} by defining the additive function 
$$\nu(x):=\nu(V_{r^{n(x)}(x),r^{n(x)-1}(x),\dots, x}),\quad(x\in\T(x_0)).$$
Uniqueness is clear since $W\neq 0$ on $\T(x_0)$. 

For the converse, we compute
$$\sum_{r(y)=x}W(y)u(y)=\sum_{r(y)=x}W(y)\frac{1}{W(y)W(r(y))\dots W(r^{n(y)-1}(y))}\nu(y)$$$$=\frac{1}{W(x)\dots W(r^{n(x)-1}(x))}\sum_{r(y)=x}\nu(y)=u(x).$$
\end{proof}

\begin{remark}\label{rem1.3}
Note that the function $\nu_0(x)=W^{(n(x))}(x)$, $x\in\T(x_0)$ is an additive function. Indeed 
$$\sum_{r(y)=x}\nu_0(y)=\sum_{r(y)=x}W^{(n(y))}(y)=\sum_{r(y)=x}W(y)W(r(y))\dots W(r^{n(y)-1}(y))$$$$=W(x)\dots W(r^{n(x)-1}(x))\sum_{r(y)=x}W(y)=\nu_0(x)\cdot 1.$$
Therefore we have the following corollary.

\end{remark}
\begin{corollary}\label{cor1.3}
Every non-negative harmonic function is the quotient of two additive functions. 
\end{corollary}

\begin{definition}\label{def1.12}
A function $U$ on $\T(x_0)$ is called a {\it QMF-weight} if $U\geq 0$ and 
$$\sum_{r(y)=x}U(y)=1,\quad(x\in\T(x_0)).$$ 
\end{definition}
\begin{proposition}\label{pr1.12}
There exists a one-to-one correspondence between positive additive functions and positive QMF-weights on $\T(x_0)$.

For every positive additive function $\nu$ on $\T(x_0)$ the function 
$$U_\nu(x)=\frac{\nu(x)}{\nu(r(x))},\quad(x\in\T(x_0)\setminus\{x_0\}),$$
is a QMF-weight.

Conversely, for every positive QMF-weight $U$, the function 
$$\nu(x)=\nu(x_0)U(x)U(r(x))\dots U(r^{n(x)-1}(x)),\quad(x\in\T(x_0)\setminus\{x_0\}),$$
is additive, where $\nu(x_0)$ is some fixed non-negative constant. 

\end{proposition}

\begin{proof}
We have
$$\sum_{r(y)=x}U_\nu(y)=\sum_{r(y)=x}\frac{\nu(y)}{\nu(r(y))}=\frac{1}{\nu(x)}\sum_{r(y)=x}\nu(y)=1.$$
The converse follows as in Remark \ref{rem1.3}.
\end{proof}

\begin{definition}\label{def1.14}
Let $X_r$ be the set of points $x_0\in X$ such that $r^n(x_0)$ is regular for all $n\geq0$. 
\end{definition}
\begin{remark}
Note that $X_r$ is invariant for both $r$ and $r^{-1}$.
\end{remark}

\begin{proposition}\label{pr1.16}
Let $h$ be an $R_W$-harmonic function on $X_r$, i.e., equation \eqref{eqrwh} is satisfied for $x\in X_r$. Then for each $x_0\in X_r$ there exists a unique additive function $\nu_{x_0}$ on $\T(x_0)$ such that 
\begin{equation}
h(x)=\frac{\nu_{x_0}(x)}{W(x)W(r(x))\dots W(r^{n_{x_0}(x)-1}(x))},\quad(x\in\T(x_0)).
\label{eqhnu}
\end{equation}
Moreover, the functions $\nu_{x_0}$ are related by 
\begin{equation}
\nu_{r(x_0)}(x)=W(x_0)\nu_{x_0}(x),\quad(x\in\T(x_0)).
\label{eqnurnu}
\end{equation}
Conversely, if $\nu_{x_0}$ is an additive function on $\T(x_0)$ for all $x_0\in X_r$, and the functions satisfy the relation \eqref{eqnurnu} then the function 
$h$ on $X_r$ defined by
\begin{equation}
h(x)=\frac{\nu_{r(x)}(x)}{W(x)}=\nu_x(x),\quad(x\in X_r),
\label{eqnuh}
\end{equation}
is $R_W$-harmonic on $X_r$.
\end{proposition}

\begin{proof}
Since the restriction of $h$ to $\T(x_0)$ is $p$-harmonic, the existence and uniqueness of $\nu_{x_0}$ such that \eqref{eqhnu} is satisfied follows from Corollary \ref{cor1.2}.

We have for $x\in\T(x_0)$, $x$ is also in $\T(r(x_0))$ so
$$\frac{\nu_{x_0}(x)}{W(x)W(r(x))\dots W(r^{n_{x_0}(x)-1}(x))}=h(x)=\frac{\nu_{r(x_0)}(x)}{W(x)\dots W(r^{n_{r(x_0)}(x)-1}(x))}.$$
Since $n_{r(x_0)}(x)=n_{x_0}(x)+1$, and $r^{n_{x_0}(x)}(x)=x_0$, we have
$$W(x_0)\nu_{x_0}(x)=\nu_{r(x_0)}(x).$$

For the converse, we compute
$$\sum_{r(y)=x}W(y)h(y)=\sum_{r(y)=x}W(y)\frac{\nu_{r(y)}(y)}{W(y)}=\sum_{r(y)=x}\nu_{x}(y)=\nu_x(x)=\frac{\nu_{r(x)}(x)}{W(x)}=h(x).$$
\end{proof}
\begin{acknowledgements}
The authors gratefully acknowledge helpful discussions with Professors Judith Packer, and Erin Pearse. We also thank the participants at an Oberwolfach workshop in the Spring of 2011.  And we thank Oberwolfach for hospitality. 
\end{acknowledgements}
\bibliographystyle{alpha}	
\bibliography{judy}

\end{document}